\documentclass[]{amsart}
\usepackage{times}
\usepackage{amsfonts}
\usepackage{amssymb}
\usepackage{bbm}
\usepackage{times}
\usepackage{amssymb}
\usepackage{amscd}
\usepackage{graphicx}

\usepackage{amsmath}
\usepackage{amssymb}

\usepackage{amsmath}
\usepackage{amsfonts}
\usepackage{amscd}
\usepackage{latexsym}
\usepackage{amsthm}
\usepackage{amssymb}
\usepackage{amsmath}

\theoremstyle{plain}
\newtheorem{theorem}{Theorem}[section]
\newtheorem{corollary}[theorem]{Corollary}
\newtheorem{lemma}[theorem]{Lemma}

\newtheorem{definition}[theorem]{Definition}

\begin{document}
\title[asymptotically tracially  approximation ${\rm C^*}$-algebras]{{\bf comparison  properties  of asymptotically tracially  approximation ${\rm C^*}$-algebras}}
\author{Qingzhai Fan and  Xiaochun Fang}

\address{ Qingzhai Fan\\ Department of Mathematics\\  Shanghai Maritime University\\
Shanghai\\China
\\  201306 }
\email{qzfan@shmtu.edu.cn}

\address{Xiaochun Fang\\ Department of Mathematics\\ Tongji
University\\
Shanghai\\China
\\  200092 }
\email{xfang@mail.tongji.edu.cn}


\begin{abstract} We show that  the following  properties  of the
  ${\rm C^*}$-algebras in a class $\mathcal{P}$ are inherited by simple unital ${\rm C^*}$-algebras in the
 class of asymptotically tracially in $\mathcal{P}$: $(1)$  $\beta$-comparison (in the sense of Kirchberg and R{\o}rdam),  $(2)$ $n$-comparison (in the sense of Winter).
 \end{abstract}
\maketitle
\section{Introduction}
      Tracial topological rank no more than $k$ was introduced by Lin in \cite{L1}. Instead of assuming inductive limit structure, he started with a
  certain abstract tracial approximation property, and 
  ${\rm C^*}$-algebras with tracial topological rank no more than one and
  certain additional properties are AH algebras without dimension
  growth  which was classified by Elliott-Gong in \cite{E6} (in the real rank zero case) and classified by Elliott-Gong-Li in \cite{GL}.

  This abstract tracial approximation structure has proved to be very important in the classification of simple amenable ${\rm C^*}$-algebras. For example,
   it led to the classification of unital simple separable amenable  ${\rm C^*}$-algebras with finite nuclear dimension in the UCT class (see \cite{G4}, \cite{EZ5}, \cite{TWW1}).

Examples of R{\o}rdam \cite{RR10} and Toms \cite{TT1}, relying on techniques pioneered by Villadsen \cite{V}, demonstrated that some sort of regularity condition, stronger than nuclearity, is necessary in order to have a classification by $K$-theory and trace. Three regularity conditions have emerged: finite nuclear dimension, tensorial absorption of the Jiang-Su algebra $\mathcal{Z},$ and algebraic regularity in the Cuntz semigroup.
Toms and  Winter have conjectured  (see e.g. \cite{ET}) that these three fundamental properties are equivalent for all separable, simple, unital, amenable ${\rm C^*}$-algebras.

 Kirchberg and R{\o}rdam introduced a weaker comparison property and also a  property  of  a ${\rm C^*}$-algebra called $\beta$-comparison  in  \cite{K22}. The  property of  $n$-comparison  was introduced by Winter in
 \cite{WW23}.

It is an open problem if  Kirchberg's and R{\o}rdam's weak comparison and $\beta$-comparison, Winter's $m$-comparison, and strict comparison all agree for simple  unital ${\rm C^*}$-algebras. (Somewhat confusingly, it is  known that $m$-comparison for a particular $m$ does agree with  $\beta$-comparison for $\beta=m+1$.)

In order to search a tracial version of Toms-Winter conjecture, also inspire by tracial $\mathcal{Z}$-absorbing ${\rm C^*}$-algebras which was introduced by Hirshberg and Orovitz in \cite{HO},  Fu and Lin  introduced a class of asymptotically tracially  approximation ${\rm C^*}$-algebras in \cite{FL}.

The following  definition is not  exactly  same as  definition 3.1 in \cite{FL}, but by proposition 3.8 in \cite{FL}, the following definition is equivalent to the  definition 3.1 in \cite{FL}.

Let  $\mathcal{P}$ be a class of unital ${\rm C^*}$-algebras,   the class of simple unital  ${\rm C^*}$-algebras which can be  asymptotically tracially in $\mathcal{P}$, denoted by ${\rm ATA}\mathcal{P}$.

\begin{definition}(\cite{FL}.)\label{def:2.5}  A  simple unital ${\rm C^*}$-algebra $A$ is said to belong to the class ${\rm ATA}\mathcal{P}$ if, for any
 $\varepsilon>0,$ any finite
subset $F\subseteq A,$ and any  non-zero element $a\geq 0,$ there exist
a ${\rm C^*}$-algebra $B$  in $\mathcal{P}$ and completely positive  contractive linear maps  $\alpha:A\to B$ and  $\beta_n: B\to A$, and $\gamma_n:A\to A\cap\beta_n(B)^{\perp}$ such that

$(1)$ the map $\alpha$ is unital  completely positive   linear map,  $\beta_n(1_B)$ and $\gamma_n(1_A)$ are projections and  $\beta_n(1_B)+\gamma_n(1_A)=1_A$ for all $n\in \mathbb{N}$,

$(2)$ $\|x-\gamma_n(x)-\beta_n(\alpha(x))\|<\varepsilon$ for all $x\in F$ and for all $n\in {\mathbb{N}}$,

$(3)$ $\alpha$ is an $F$-$\varepsilon$ approximate embedding,

$(4)$ $\lim_{n\to \infty}\|\beta_n(xy)-\beta_n(x)\beta_n(y)\|=0$ and $\lim_{n\to \infty}\|\beta_n(xy)\|=\|x\|$ for all $x,y\in B$, and

$(5)$ $\gamma_n(1)\lesssim a$ for all $n\in \mathbb{N}$.
\end{definition}

 In \cite{FL}, Fu and Lin show that  the following  properties  of unital
${\rm C^*}$-algebras in a class $\mathcal{P}$ are inherited by simple unital ${\rm C^*}$-algebras in the  class  ${\rm ATA}\mathcal{P}$:
  $(1)$ stably finite,
   $(2)$ quasidiagonal ${\rm C^*}$-algebras,
    $(3)$ purely infinite simple ${\rm C^*}$-algebras, $(4)$ tracial $\mathcal{Z}$-absorption,    $(5)$ the Cuntz semigroup is almost unperforated, and   $(6)$ strict comparison property.

In this paper, we show that  the following comparison  properties  of unital
${\rm C^*}$-algebras in a class $\mathcal{P}$ are inherited by simple unital ${\rm C^*}$-algebras in  the class of asymptotically tracially in $\mathcal{P}$:
  $(1)$ $\beta$-comparison ( in the sense of Kirchberg and R{\o}rdam; see \cite{K22}),
   $(2)$ $n$-comparison (in the sense of Winter; see \cite{WW23}).

\section{Definitions and preliminaries}

 Let  ${\rm M}_{n}(A)_+$ denote
the positive elements of  ${\rm M}_{n}(A)$.  Given $a, b\in {\rm M}_{n}(A)_+,$
we say that $a$ is Cuntz subequivalent to $b$ (written $a\lesssim b$) if there is a sequence $(v_n)_{n=1}^\infty$
of elements of ${\rm M}_{n}(A)$ such that $$\lim_{n\to \infty}\|v_nbv_n^*-a\|=0.$$
We say that $a$ and $b$ are Cuntz equivalent (written $a\sim b$) if $a\lesssim b$ and $b\lesssim a$. We write $\langle a\rangle$ for the equivalence class of $a$.

The object ${\rm W}(A):={\rm M}_{\infty}(A)_+/\sim$
 will be called the Cuntz semigroup of $A$. (See \cite{CEI}.)   ${\rm W}(A)$ becomes  an ordered  semigroup   when equipped with the addition operation
$$\langle a\rangle+\langle b\rangle=\langle a \oplus b\rangle,$$
 and the order relation
$$\langle a\rangle\leq \langle b\rangle\Leftrightarrow a\lesssim b.$$

Given $a$ in ${\rm M}_{\infty}(A)_+$ and $\varepsilon>0,$ we denote by $(a-\varepsilon)_+$ the element of ${\rm C^*}(a)$ corresponding (via the functional calculus) to the function $f(t)={\max (0, t-\varepsilon)},  ~ t\in \sigma(a)$.

We shall say that a separable exact ${\rm C^*}$-algebra $A$ has strict comparison if for $a, ~b\in {\rm M}_k(A)_+,$ with $d_\tau(a)<d_\tau(b)$ for any $\tau\in {\rm T}(A),$ then we have $a\lesssim b$, where ${\rm T}(A)$ is  the set of tracial states of $A$.

\begin{definition} (\cite{WW23}) Let $A$ be a unital ${\rm C^*}$-algebra. We say $A$  has $n$-comparison,  if,  whenever  $x,~y_0,~y_1, $ $~y_2,$
$\cdots,~y_n$ are elements
 in ${\rm W}(A)$ such that $ x<_s  y_j$ for all $j=0,~ 1,~ \cdots, ~ n,$ then  $x\leq y_0+y_1+\cdots +y_n$. Here, $ x<_s  y$ means  $(k+1)x\leq ky$
 for some natural number $k$.
\end{definition}

\begin{definition} (\cite{K22}) Let $A$ be a unital ${\rm C^*}$-algebra and let $1\leq \beta < \infty$. We say that $A$ has
$\beta$-comparison if for all $x,  y\in $ ${\rm W}(A)$ and all integers $k,~l\geq 1$ with $k>\beta l,$ the inequality $kx\leq ly$ implies $x\leq y$.

It follows immediately from the definitions that ${\rm W}(A)$ is almost unperforated if and only if $A$ has 1-comparison in the sense of Kirchberg and R{\o}rdam.
\end{definition}

\begin{theorem}(\cite{PPT}, \cite{HO}.) \label{thm:2.1} Let $A$ be a stably finite ${\rm C^*}$-algebra.

 $(1)$ Let $a,~ b\in A_+$ and   $\delta>0$  be such that
$\|a-b\|<\delta$. Then we have $(a-\delta)_+\lesssim b$.

$(2)$ Let $a$ be a purely positive element of $A$ (i.e., $a$ is
not Cuntz equivalent to a projection). Let  $\eta>0$,  and let $f\in C_0(0,1]$ be a non-negative function with $f=0$ on $(\eta,1),$  $f>0$ on $(0,\eta)$,
and $\|f\|=1$.  We have $f(a)\neq 0$
and  $(a-\eta)_++f(a)\lesssim a.$

$(3)$ The following conditions are equivalent: $(1)'$ $a\lesssim b,$ $(2)'$ for any $\delta>0,$ $(a-\delta)_+\lesssim b,$ and  $(3)'$ for any $\delta>0$,  there is $\varepsilon>0,$ such that $(a-\delta)_+\lesssim (b-\varepsilon)_+.$
 \end{theorem}

\begin{lemma}(\cite{FL}.)\label{lem:2.6} If the class $\mathcal{P}$ is closed under tensoring with matrix algebras and
 under passing to  unital
hereditary ${\rm C^*}$-subalgebras, then  the class ${\rm ATA}\mathcal{P}$ is closed under
tensoring with matrix algebras  and under passing to unital hereditary ${\rm C^*}$-subalgebras.
\end{lemma}

The following lemma is obvious,  and we omit the proof.
\begin{lemma}\label{lem:2.7} The   $n$-comparison (or $\beta$-comparison) is preserved  under  tensoring with
matrix algebras and  under passing to  unital hereditary ${\rm C^*}$-subalgebras.
\end{lemma}

\section{The main results}

\begin{theorem}\label{thm:3.1}
 Let $\mathcal{P}$ be a class of  stably finite unital
${\rm C^*}$-algebras which have  $\beta$-comparison (in the sense of Kirchberg and R{\o}rdam), for some $1\leq \beta<\infty$.  Then $A$  has  $\beta$-comparison  for  any  simple  unital ${\rm C^*}$-algebra $A\in {\rm ATA}\mathcal{P}.$
\end{theorem}

\begin{proof}
By Lemma \ref{lem:2.7} and Lemma  \ref{lem:2.6},  enlarging the class $\mathcal{P}$, we may suppose it is closed
under passing to matrix algebras and unital hereditary ${\rm C^*}$-subalgebras (i.e., Morita equivalent  ${\rm C^*}$-algebras).

  Let $a, b\in {\rm M}_{\infty}(A)_+$. By Theorem \ref {thm:2.1} (3),
  we need to show that $\langle(a-\varepsilon)_+\rangle\leq\langle b\rangle$ for any $\varepsilon>0,$  and  any  integers
 $k,~l\geq 1$ such that   $k>\beta l,$ $k \langle a\rangle \leq   l\langle b\rangle.$

We may assume that $a, b\in {\rm M}_{n}(A)_+$ for some  integer $n$. By Lemma \ref {lem:2.6} and Lemma  \ref{lem:2.7}, we may assume that
$a, b\in A_+$ and $\|a\|\leq 1,\|b\|\leq 1$.

    We divide the proof into three cases.

   $(\textbf{I})$, we suppose that  $b$ is not Cuntz equivalent to a projection.

 Given $\delta>0$, and $k,l\geq 1$ as above, with in particular $k>l$, since  $k\langle a\rangle \leq l\langle b\rangle.$
    Hence,
    by Theorem  \ref {thm:2.1} (1), there exists
  $v=(v_{i,j})\in {\rm M}_{k}(A),  ~1\leq i\leq k,~ 1\leq j\leq k$ such that
 $$ \|v({\rm diag}(b\otimes 1_{l},~0\otimes 1_{k-l})){v}^*-a\otimes 1_{k}\|<\delta.$$

By Theorem \ref {thm:2.1} (2),  there is a non-zero positive element $d$ orthogonal to $b$  such that  $$(b-\delta/2)_++d\lesssim b.$$

 With  $F=\{a,~ b, ~ v_{i,j}: ~1\leq i\leq k,~ 1\leq j\leq k\},$  and  $\varepsilon'>0$, since $A\in {\rm ATA}\mathcal{P},$ there exist
a ${\rm C^*}$-algebra $B$  in $\mathcal{P}$ and completely positive  contractive linear maps  $\alpha:A\to B$ and  $\beta_n: B\to A$, and $\gamma_n:A\to A\cap\beta_n(B)^{\perp}$ such that

$(1)$ the map $\alpha$ is unital  completely positive   linear map, $\beta_n(1_B)$ and $\gamma_n(1_A)$ are all projections, and  $\beta_n(1_B)+\gamma_n(1_A)=1_A$ for all $n\in \mathbb{N}$,

$(2)$ $\|x-\gamma_n(x)-\beta_n(\alpha(x))\|<\varepsilon'$ for all $x\in F$ and for all $n\in {\mathbb{N}}$,

$(3)$ $\alpha$ is an $F$-$\varepsilon'$ approximate embedding,

$(4)$ $\lim_{n\to \infty}\|\beta_n(xy)-\beta_n(x)\beta_n(y)\|=0$ and $\lim_{n\to \infty}\|\beta_n(xy)\|=\|x\|$ for all $x,y\in B$, and

$(5)$ $\gamma_n(1)\lesssim d$ for all $n\in \mathbb{N}$.

Since $ \|v{\rm diag}(b\otimes 1_{l},~0\otimes 1_{k-l}){v}^*-a\otimes 1_{k}\|<\delta$,
 by $(1)$,   we have
 $$ \|\alpha\otimes id_{M_k}(v({\rm diag}(b\otimes 1_{l},~0\otimes 1_{k-l})){v}^*)-\alpha(a)\otimes 1_{k}\|<\delta.$$
  By $(3)$, we have
 $$ \|\alpha\otimes id_{M_k}(v)({\rm diag}(\alpha(b)\otimes 1_{l},~0\otimes 1_{k-l}))\alpha\otimes id_{M_k}({v}^*)-\alpha(a)\otimes 1_{k}\|<2\delta.$$
 By  Theorem  \ref {thm:2.1} (1), we have
 $$k\langle(\alpha(a)-2\delta)_+\rangle\leq l\langle \alpha(b)\rangle.$$
 Since $B\in \Omega$, we have $$\langle(\alpha(a)-2\delta)_+\rangle\leq \langle \alpha(b)\rangle.$$
   Since $\langle(\alpha(a)-2\delta)_+\rangle\leq \langle \alpha(b)\rangle$, there exist $w\in B$ such that $$\|w\alpha(b)w^*- (\alpha(a)-2\delta)_+\|<\delta.$$
   Since $\|w\alpha(b)w^*- (\alpha(a)-2\delta)_+\|<\delta,$  we have
   $$\|\beta_n(w\alpha(b)w^*)- \beta_n((\alpha(a)-2\delta)_+)\|<\delta.$$
   By $(4)$ we have $$\|\beta_n(w)\beta_n(\alpha(b))\beta_n(w^*)- \beta_n((\alpha(a)-2\delta)_+)\|<2\delta.$$
 By Theorem  \ref {thm:2.1} (1), we have $$\langle(\beta_n((\alpha(a)-6\delta)_+)-2\delta)_+\rangle\leq \langle (\beta_n\alpha(b)-2\delta)_+\rangle.$$
 Therefore, we have
\begin{eqnarray}
\label{Eq:eq1}
&&\langle(a-\varepsilon)_+\rangle \nonumber\\
&&\leq\langle\gamma_n(a)\rangle+\langle (\beta_n\alpha(a)-20\delta)_+\rangle\nonumber\\
&&\leq\langle\gamma_n(1_A)\rangle+\langle (\beta_n\alpha(a)-20\delta)_+)\rangle\nonumber\\
&&\leq\langle d\rangle+\langle ((\beta_n(\alpha(a)-2\delta)_+)-6\delta)_+\rangle\nonumber\\
&&\leq\langle (b-\delta/2)_+\rangle
+\langle d\rangle\leq\langle b\rangle.\nonumber
\end{eqnarray}

$(\textbf{II})$, we suppose that  $b$ is Cuntz equivalent to a projection and   $a$ is not Cuntz equivalent to a projection.
Choose a projection $p_0$ such that $b$ is  Cuntz equivalent to $p_0.$
We may assume that $b=p_0.$

By Theorem \ref {thm:2.1} (2), let $\delta>0$ (with $\delta<\varepsilon$),  there exists a non-zero positive element $d$  orthogonal to $a$  such that
 $\langle(a-\delta)_++ d\rangle \leq\langle a\rangle.$
Since $k\langle a\rangle \leq l\langle p_0\rangle,$ we have $k\langle(a-\delta)_+ +d\rangle \leq l\langle p_0\rangle.$ Hence, as above, by
 Theorem \ref {thm:2.1} (1),  on increasing $\delta$ slightly, and decreasing $d$ slightly, there exists  $v=(v_{i,j})\in {\rm M}_{k}(A), ~1\leq i\leq k,~ 1\leq j\leq k$, such that
   $$\|v{\rm diag}(p_0\otimes 1_{l}, ~0\otimes 1_{k-l}) v^*-((a-\delta)_++ d)\otimes1_{k}\|<\delta.$$

  Since $A\in {\rm ATA}\mathcal{P},$  for $F=\{a,~ p_0, ~d,~ {v_{i,j}}: ~1\leq i\leq k,~ 1\leq j\leq k\},$  and  $\varepsilon'>0$, there exist
a ${\rm C^*}$-algebra $B$  in $\mathcal{P}$ and completely positive  contractive linear maps  $\alpha:A\to B$ and  $\beta_n: B\to A$, and $\gamma_n:A\to A\cap\beta_n(B)^{\perp}$ such that

$(1)$ the map $\alpha$ is unital  completely positive   linear map, $\beta_n(1_B)$ and $\gamma_n(1_A)$ are all projections, $\beta_n(1_B)+\gamma_n(1_A)=1_A$ for all $n\in \mathbb{N}$,

$(2)$ $\|x-\gamma_n(x)-\beta_n(\alpha(x))\|<\varepsilon'$ for all $x\in F$ and for all $n\in {\mathbb{N}}$.

$(3)$ $\alpha$ is an $F-\varepsilon'$-approximate embedding,

$(4)$ $\lim_{n\to \infty}\|\beta_n(xy)-\beta_n(x)\beta_n(y)\|=0$ and $\lim_{n\to \infty}\|\beta_n(xy)\|=\|x\|$ for all $x,y\in B$.

Since $ \|v{\rm diag} (p_0\otimes 1_{l}, ~0\otimes 1_{k-l}) v^*-((a-\delta)_++ d)\otimes1_{k}\|<\delta$,
  by $(1)$,   we have
 $$ \|\alpha\otimes id_{M_k}(v({\rm diag}(p_0\otimes 1_{l},~0\otimes 1_{k-l})){v}^*)-\alpha((a-\delta)_++ d)\otimes 1_{k}\|<\delta.$$
  By $(3)$, we have
 $$ \|\alpha\otimes id_{M_k}(v)({\rm diag}(\alpha(p_0)\otimes 1_{l},~0\otimes 1_{k-l}))\alpha\otimes id_{M_k}({v}^*)-\alpha((a-\delta)_++ d)\otimes 1_{k}\|<2\delta.$$
 By  Theorem  \ref {thm:2.1} (1), we have
 $$k\langle(\alpha((a-\delta)_++ d)-2\delta)_+\rangle\leq l\langle \alpha(b)\rangle.$$
 Since $B\in \Omega$, we have $$\langle(\alpha((a-\delta)_++ d)-2\delta)_+\rangle\leq \langle \alpha(b)\rangle.$$
   Since $\langle(\alpha((a-\delta)_++ d)-2\delta)_+\rangle\leq \langle \alpha(b)\rangle$, there exist $w\in B$ such that $$\|w\alpha(b)w^*- (\alpha((a-\delta)_++ d)-2\delta)_+\|<\delta.$$
   Since $\|w\alpha(b)w^*- (\alpha((a-\delta)_++ d)-2\delta)_+\|<\delta$,  we have
   $$\|\beta_n(w\alpha(b)w^*)- \beta_n((\alpha((a-\delta)_++ d)-2\delta)_+)\|<\delta.$$
   By $(4)$ we have $$\|\beta_n(w)\beta_n\alpha(b)\beta_n(w^*)- \beta_n((\alpha((a-\delta)_++ d)-2\delta)_+)\|<2\delta.$$
By Theorem \ref {thm:2.1} (1), we have $$\langle(\beta_n(\alpha((a-\delta)_++ d))-6\delta)_+\rangle\leq \langle \beta_n\alpha(b)\rangle.$$

Since $(a-\delta)_+$ orthogonal to $d$, by $(3)$ and
$(4)$, we may assume that $\beta_n\alpha((a-\delta)_+)$ orthogonal to $\beta_n\alpha(d)$.

 With  $G=\{\gamma_n(a), \gamma_n(p_0),~  \gamma_n(v_{i,j}): ~1\leq i\leq k,~ 1\leq j\leq k\},$  and any $\varepsilon''>0$, let $E=\gamma_n(1)A\gamma_n(1)$, since $E$ is  asymptotically tracially in $\Omega$, there exist
a ${\rm C^*}$-algebra $D$  in $\Omega$ and completely positive  contractive linear maps  $\alpha':E\to D$ and  $\beta_n': D\to E$, and $\gamma_n':E\to E\cap\beta_n'(D)^{\perp}$ such that

$(1)'$ the map $\alpha'$ is unital  completely positive   linear map, $\beta_n'(1_D)$ and $\gamma_n'(1_E)$ are all projections, $\beta_n'(1_D)+\gamma_n'(1_E)=1_E$ for all $n\in \mathbb{N}$,

$(2)'$ $\|x-\gamma_n'(x)-\beta_n'(\alpha'(x))\|<\varepsilon''$ for all $x\in G$ and for all $n\in {\mathbb{N}}$,

$(3)'$ $\alpha'$ is an $F$-$\varepsilon''$ approximate embedding,

$(4)'$ $\lim_{n\to \infty}\|\beta_n'(xy)-\beta_n'(x)\beta_n(y)\|=0$ and $\lim_{n\to \infty}\|\beta_n'(xy)\|=\|x\|$ for all $x,y\in D$, and

$(5)'$ $\gamma_n'(1)\lesssim \beta_n\alpha(d)$ for all $n\in \mathbb{N}$.

By $(2)$, we have
 \begin{eqnarray}
\label{Eq:eq1}
&&\|(\gamma_n\otimes id_{M_{k}}(v))({\rm diag}(\gamma_n(p_0)\otimes 1_l, ~0\otimes 1_{k-l}))(\gamma_n\otimes id_{M_{k}}({v}^*))\nonumber\\
&&-\gamma_n((a-\delta)_++d)\otimes 1_k\|
<\delta.\nonumber
\end{eqnarray}

Since $ \|(\gamma_n\otimes id_{M_{k}}(v))({\rm diag}(\gamma_n(p_0)\otimes 1_l, ~0\otimes 1_{k-l}))(\gamma_n\otimes id_{M_{k}}({v}^*))
-\gamma_n((a-\delta)_++d)\otimes 1_k\|
<\delta$,
  by $(1)'$ and $(3)'$,   we have
 \begin{eqnarray}
\label{Eq:eq1}
&&\|\alpha'\otimes id_{M_{k}}\gamma_n\otimes id_{M_{k}}(v){\rm diag}(\alpha'\gamma_n(p_0)\otimes 1_l, ~0\otimes 1_{k-l}))\nonumber\\
&&\alpha'\otimes id_{M_{k}}\gamma_n\otimes id_{M_{k}}(v^*)-\alpha'\gamma_n((a-\delta)_++d)\otimes 1_k\|<2\delta.\nonumber
\end{eqnarray}
 By Theorem \ref {thm:2.1} (1),  we have
  $$k\langle(\alpha'\gamma_n((a-\delta)_++d)-4\delta)_+\rangle\leq l\langle (\alpha'\gamma_n(p_0)-\delta)_+\rangle. $$
Since $B\in \Omega,$  this implies  $$\langle (\alpha'\gamma_n((a-\delta)_++d)-4\delta)_+\rangle\leq \langle (\alpha'\gamma_n(p_0)-\delta)_+\rangle.$$
Since $\langle (\alpha'\gamma_n((a-\delta)_++d)-4\delta)_+\rangle\leq \langle (\alpha'\gamma_n(p_0)-\delta)_+\rangle,$ there exist $w\in D$ such that
$$\|w(\alpha'\gamma_n(p_0)-\delta)_+w^*-\alpha'\gamma_n((a-\delta)_++d)-4\delta)_+\|<\delta.$$
By $(4)'$, we have
$$\|\beta_n'(w)\beta_n'\alpha'\gamma_n(p_0)-\delta)_+\beta_n(w^*)-
\beta_n\alpha'\gamma_n((a-\delta)_++d)-\delta)_+\|<2\delta.$$
 By Theorem \ref {thm:2.1} (1),  we have
$$\langle (\beta_n'(\alpha'\gamma_n((a-\delta)_++d)-\delta)_+-2\delta)_+\rangle\leq \langle \beta_n'(\alpha'\gamma_n(p_0)-\delta)_+)\rangle.$$

Therefore, we have
\begin{eqnarray}
\label{Eq:eq1}
&&\langle(a-\varepsilon)_+\rangle \nonumber\\
&&\leq\langle(\gamma_n(a)-8\delta)_+\rangle+\langle (\beta_n\alpha((a-\delta)_+)-2\delta)_+\rangle\nonumber\\
 &&\leq\langle\gamma_n'\gamma_n(a)\rangle+\langle \beta_n'\alpha'\gamma_n((a-\delta)_+-6\delta)_+))\rangle\nonumber\\
  &&+\langle (\beta_n\alpha((a-\delta)_+)-2\delta)_+\rangle\leq\langle\gamma_n'\gamma_n(1_A)\rangle+
  \langle (\beta_n\alpha((a-\delta)_+)-2\delta)_+\rangle \nonumber\\
&&+\langle (\beta_n'\alpha'\gamma_n((a-\delta)_+)-2\delta)_+\rangle\nonumber\\
&&\leq\langle\beta_n\alpha(d)\rangle+
  \langle (\beta_n\alpha((a-\delta)_+)-2\delta)_+\rangle+\langle \beta_n'\alpha'\gamma_n((a-\delta)_+)-4\delta)_+\rangle\nonumber\\
&&\leq\langle(\beta_n'\alpha'\gamma_n(p_0)-\delta)_+\rangle
+\langle \beta_n\alpha(p_0)\rangle\leq\langle p_0\rangle.\nonumber
\end{eqnarray}

$(\textbf{III})$, we suppose that both $a$ and $b$ are  Cuntz equivalent to projections.

   Choose  projections $p,~ q$ such that $a$ is Cuntz equivalent to $p$ and  $b$  is Cuntz equivalent to $q.$
We may assume that $a=p,~ b=q.$ Since
 $k\langle p\rangle\leq l\langle q\rangle$. Hence,
    by Theorem  \ref {thm:2.1} (1), there exists
  $v=(v_{i,j})\in {\rm M}_{k}(A),  ~1\leq i\leq k,~ 1\leq j\leq k$ such that
 $$ v{\rm diag}(q\otimes 1_{l},~0\otimes 1_{k-l}){v}^*=p\otimes 1_{k}.$$
Since $A\in {\rm ATA}\mathcal{P}$, for $F=\{p,~ q, \gamma_n(v_{i,j}): ~1\leq i\leq k,~ 1\leq j\leq k\},$ for  any $\varepsilon'>0,$
  there exist
a ${\rm C^*}$-algebra $B$  in $\mathcal{P}$ and completely positive  contractive linear maps  $\alpha:A\to B$ and  $\beta_n: B\to A$, and $\gamma_n:A\to A\cap\beta_n(B)^{\perp}$ such that

$(1)$ the map $\alpha$ is unital  completely positive   linear map, $\beta_n(1_B)$ and $\gamma_n(1_A)$ are all projections $\beta_n(1_B)+\gamma_n(1_A)=1_A$ for all $n\in \mathbb{N}$,

$(2)$ $\|x-\gamma_n(x)-\beta_n(\alpha(x))\|<\varepsilon'$ for all $x\in F$ and for all $n\in {\mathbb{N}}$.

$(3)$ $\alpha$ is an $F$-$\varepsilon'$ approximate embedding,

$(4)$ $\lim_{n\to \infty}\|\beta_n(xy)-\beta_n(x)\beta_n(y)\|=0$ and $\lim_{n\to \infty}\|\beta_n(xy)\|=\|x\|$ for all $x,y\in B$.

Since $ v{\rm diag} (q\otimes 1_{l}, ~0\otimes 1_{k-l}) v^*=p\otimes1_{k}$,
by $(2)$, we have
 \begin{eqnarray}
\label{Eq:eq1}
&&\|(\gamma_n\otimes id_{M_{k}}(v)+\beta_n\otimes id_{M_{k}}(\alpha\otimes id_{M_{k}}(v)))({\rm diag}(\gamma_n(q)\otimes 1_l, ~0\otimes 1_{k-l}) \nonumber\\
&&+{\rm diag}(\beta_n\alpha(q)\otimes 1_l, ~0\otimes 1_{k-l}))(\gamma_n\otimes id_{M_{k}}({v}^*)+\beta_n\otimes id_{M_{k}}(\alpha\otimes id_{M_{k}}({v}^*)))\nonumber\\
&&-(\gamma_n(p)\otimes 1_k, +\beta_n\alpha(p)\otimes 1_k)\|<\delta.\nonumber
\end{eqnarray}

 By $(1)$, we have
 \begin{eqnarray}
\label{Eq:eq1}
&&\|(\gamma_n\otimes id_{M_{k}}(v)){\rm diag}(\gamma_n(q)\otimes 1_l, ~0\otimes 1_{k-l})(\gamma_n\otimes id_{M_{k}}({v}^*))\nonumber\\
&&-\gamma_n(p)\otimes 1_k\|
<\delta.\nonumber
\end{eqnarray}
Since $ v{\rm diag} (q\otimes 1_{l}, ~0\otimes 1_{k-l}) v^*=p\otimes1_{k}$,
by $(1)$ and $(3)$, we have
$$\|\alpha\otimes id_{M_{k}}(v)){\rm diag}(\alpha(q)\otimes 1_l, ~0\otimes 1_{k-l})
\alpha\otimes id_{M_{k}}({v}^*)-\alpha(p)\otimes 1_k\|
<\delta.$$
Since $\|\alpha\otimes id_{M_{k}}(v)){\rm diag}(\alpha(q)\otimes 1_l, ~0\otimes 1_{k-l})
\alpha\otimes id_{M_{k}}({v}^*)-\alpha(p)\otimes 1_k\|
<\delta,$
by Theorem  \ref {thm:2.1} (1), we have
$$k\langle(\alpha(p)-\delta)_+\rangle\leq l\langle\alpha(q)\rangle.$$
Since $B\in \Omega$, we have
$$\langle(\alpha(p)-\delta)_+\rangle\leq \langle\alpha(q)\rangle.$$
Since $\langle(\alpha(p)-\delta)_+\rangle\leq \langle\alpha(q)\rangle,$ there exists $w\in B$ such that $$\|w\alpha(q)w^*-(\alpha(p)-\delta)_+\|<\delta.$$
By $(4)$, we have
$$\|\beta_n(w)\beta_n\alpha(q)\beta_n(w^*)-\beta_n((\alpha(p)-\delta)_+)\|<2\delta.$$
By Theorem \ref {thm:2.1} (1), we  have
$$\langle(\beta_n \alpha(p)-2\delta)_+\rangle\leq \langle \beta_n\alpha(q)\rangle.$$

$(\textbf{III.I})$, if $(\beta_n\alpha (p)-2\delta)_+$ and $\beta_n\alpha(q)$  are not Cuntz equivalent to a pure positive element, then there exist projections $p_1, q_1$ such that $(\beta_n\alpha (p)-2\delta)_+\sim p_1$ and $\beta_n\alpha(q)\sim q_1$,
 then we have $k\langle p_1\rangle\leq l\langle q_1\rangle=l\langle p_1\rangle\leq k\langle p_1\rangle.$
So $\bigoplus_{n=1}^{k}p_1$ is equivalent to a proper subprojection of itself, and this
contradicts  the stable finiteness of $A$ (since ${\rm C^*}$-algebras in $\mathcal{P}$ are stably finite (cf. proposition 4.2 in \cite{FL}).
So  we may assume that there exist a nonzero projection $s\in A$ and orthogonal to $(\beta_n\alpha (p)-2\delta)_+$ such that $(\beta_n\alpha (p)-2\delta)_++s\lesssim \beta_n\alpha(q)$.

With  $G=\{\gamma_n(p), \gamma_n(q),~  \gamma_n(v_{i,j}): ~1\leq i\leq k,~ 1\leq j\leq k\},$  and any $\varepsilon''>0$, $E=\gamma_n(1)A\gamma_n(1)$, since $E$ is  asymptotically tracially in $\Omega$, there exist
a ${\rm C^*}$-algebra $D$  in $\Omega$ and completely positive  contractive linear maps  $\alpha':E\to D$ and  $\beta_n': D\to E$, and $\gamma_n':E\to E\cap\beta_n'(D)^{\perp}$ such that

$(1)'$ the map $\alpha'$ is unital  completely positive   linear map, $\beta_n'(1_D)$ and $\gamma_n'(1_E)$ are all projections, $\beta_n'(1_D)+\gamma_n'(1_E)=1_E$ for all $n\in \mathbb{N}$,

$(2)'$ $\|x-\gamma_n'(x)-\beta_n'(\alpha'(x))\|<\varepsilon''$ for all $x\in G$ and for all $n\in {\mathbb{N}}$,

$(3)'$ $\alpha'$ is an $F$-$\varepsilon''$ approximate embedding,

$(4)'$ $\lim_{n\to \infty}\|\beta_n'(xy)-\beta_n'(x)\beta_n(y)\|=0$ and $\lim_{n\to \infty}\|\beta_n'(xy)\|=\|x\|$ for all $x,y\in D$, and

$(5)'$ $\gamma_n'\gamma_n(1)\lesssim s$ for all $n\in \mathbb{N}$.

Since
$$\|(\gamma_n\otimes id_{M_{k}}(v)){\rm diag}(\gamma_n(q)\otimes 1_l, ~0\otimes 1_{k-l})(\gamma_n\otimes id_{M_{k}}({v}^*))
-\gamma_n(p)\otimes 1_k\|
<\delta.$$

By $(1)'$, we have
\begin{eqnarray}
\label{Eq:eq1}
&&\|\alpha'\otimes id_{M_{k}}\gamma_n\otimes id_{M_{k}}(v)){\rm diag}(\alpha'\gamma_n(q)\otimes 1_l, ~0\otimes 1_{k-l})\nonumber\\
&&\alpha'\otimes id_{M_{k}}\gamma_n\otimes id_{M_{k}}(v^*)-\alpha'\gamma_n(p)\otimes 1_k\|
<2\delta.\nonumber
\end{eqnarray}
By Theorem \ref {thm:2.1} (1), we  have
$$k\langle(\alpha'\gamma_n(p)-4\delta)_+\rangle\leq l\langle (\alpha'\gamma_n(q)-\delta)_+
\rangle. $$
Since $D\in \Omega$, we have
$$\langle(\alpha'\gamma_n(p)-4\delta)_+\rangle\leq \langle (\alpha'\gamma_n(q)-\delta)_+
\rangle. $$
There exists $w\in D$ such that
$$\|w(\alpha'\gamma_n(q)-\delta)_+w^*-(\alpha'\gamma_n(p)-2\delta)_+\|<\delta.$$
By $(4')$, we have
$$\|\beta_n'(w)\beta_n'(\alpha'\gamma_n(q)-\delta)_+\beta_n(w^*)
-\beta_n((\alpha'\gamma_n(p)-2\delta)_+)\|<2\delta.$$
We  have   $$\langle (\beta_n'\alpha'\gamma_n(p)-6\delta)_+\rangle\leq \langle (\beta_n'\alpha'\gamma_n(q)-\delta)_+\rangle.$$

Therefore, if $\varepsilon'$, are small enough, then
\begin{eqnarray}
\label{Eq:eq1}
&&\langle (p-\varepsilon)_+\rangle \nonumber\\
&&\leq\langle\gamma_n(p)-6\delta)_+\rangle+\langle (\beta_n\alpha(p)-2\delta)_+\rangle\nonumber\\
 &&\leq\langle(\gamma_n'\gamma_n(p)-4\delta)_+\rangle+\langle (\beta_n'\alpha'\gamma_n(p)\rangle\nonumber\\
  &&+\langle (\beta_n\alpha(p)-2\delta)_+\rangle\leq\langle\gamma_n'\gamma_n(1_A)\rangle+
  \langle (\beta_n\alpha(p)-2\delta)_+\rangle \nonumber\\
&&+\langle (\beta_n'\alpha'\gamma_n(p)-4\delta)_+\rangle\nonumber\\
&&\leq\langle s\rangle+
  \langle (\beta_n\alpha(p)-2\delta)_+\rangle+\langle (\beta_n'\alpha'\gamma_n(p)-2\delta)_+\rangle
\leq \langle q\rangle.\nonumber
\end{eqnarray}

$(\textbf{III.II})$, We suppose that $(\beta_n\alpha (p)-2\delta)_+$ is a purely positive element, then, by Theorem \ref {thm:2.1} (2),   there is a non-zero positive element $d$ orthogonal to $(\beta_n\alpha (p)-2\delta)_+$  such that
  $((\beta_n\alpha (p)-2\delta)_+-\delta)_++d\lesssim (\beta_n\alpha (p)-2\delta)_+$.

 With  $G=\{\gamma_n(p), \gamma_n(q), ~ \gamma_n(v_{i,j}): ~1\leq i\leq k,~ 1\leq j\leq k\},$  and  $\varepsilon''>0$, $E=\gamma_n(1)A\gamma_n(1)$, since $E$ is  asymptotically tracially in $\Omega$, there exist
a ${\rm C^*}$-algebra $D$  in $\Omega$ and completely positive  contractive linear maps  $\alpha':E\to D$ and  $\beta_n': D\to E$, and $\gamma_n':E\to E\cap\beta_n'(D)^{\perp}$ such that

$(1)'$ the map $\alpha'$ is unital  completely positive   linear map, $\beta_n'(1_D)$ and $\gamma_n'(1_E)$ are all projections, $\beta_n'(1_D)+\gamma_n'(1_E)=1_E$ for all $n\in \mathbb{N}$,

$(2)'$ $\|x-\gamma_n'(x)-\beta_n'(\alpha'(x))\|<\varepsilon''$ for all $x\in G$ and for all $n\in {\mathbb{N}}$,

$(3)'$ $\alpha'$ is an $F$-$\varepsilon''$ approximate embedding,

$(4)'$ $\lim_{n\to \infty}\|\beta_n'(xy)-\beta_n'(x)\beta_n(y)\|=0$ and $\lim_{n\to \infty}\|\beta_n'(xy)\|=\|x\|$ for all $x,y\in D$, and

$(5)'$ $\gamma_n'\gamma_n(1)\lesssim d$ for all $n\in \mathbb{N}$.

Since
$$\|(\gamma_n\otimes id_{M_{k}}(v)){\rm diag}(\gamma_n(q)\otimes 1_l, ~0\otimes 1_{k-l})(\gamma_n\otimes id_{M_{k}}({v}^*))
-\gamma_n(p)\otimes 1_k\|
<\delta.$$

By $(1)'$, we have
\begin{eqnarray}
\label{Eq:eq1}
&&\|\alpha'\otimes id_{M_{k}}\gamma_n\otimes id_{M_{k}}(v){\rm diag}(\alpha'\gamma_n(q)\otimes 1_l, ~0\otimes 1_{k-l})\nonumber\\
&&\alpha'\otimes id_{M_{k}}\gamma_n\otimes id_{M_{k}}(v^*)-\alpha'\gamma_n(p)\otimes 1_k\|<2\delta.\nonumber
\end{eqnarray}
By Theorem \ref {thm:2.1} (1), we  have
$$k\langle(\alpha'\gamma_n(p)-4\delta)_+\rangle\leq l\langle (\alpha'\gamma_n(q)-\delta)_+
\rangle. $$
Since $D\in \Omega$, we have
$$\langle(\alpha'\gamma_n(p)-4\delta)_+\rangle\leq \langle (\alpha'\gamma_n(q)-\delta)_+
\rangle. $$
There exists $w\in D$ such that
$$\|w(\alpha'\gamma_n(q)-\delta)_+w^*-(\alpha'\gamma_n(p)-4\delta)_+\|<\delta.$$
By $(4')$, we have
$$\|\beta_n'(w)\beta_n'(\alpha'\gamma_n(q)-\delta)_+\beta_n(w^*)
-\beta_n((\alpha'\gamma_n(p)-4\delta)_+)\|<2\delta.$$
We  have   $$\langle (\beta_n'\alpha'\gamma_n(p)-8\delta)_+\rangle\leq \langle (\beta_n'\alpha'\gamma_n(q)-\delta)_+\rangle.$$

Therefore, if $\varepsilon'$, are small enough, then
\begin{eqnarray}
\label{Eq:eq1}
&&\langle (p-\varepsilon)_+\rangle \nonumber\\
&&\leq\langle\gamma_n(p)-6\delta)_+\rangle+\langle (\beta_n\alpha(p)-2\delta)_+\rangle\nonumber\\
 &&\leq\langle(\gamma_n'\gamma_n(p)-4\delta)_+\rangle+\langle (\beta_n'\alpha'\gamma_n(p)\rangle\nonumber\\
  &&+\langle (\beta_n\alpha(p)-2\delta)_+\rangle\leq\langle\gamma_n'\gamma_n(1_A)\rangle+
  \langle (\beta_n\alpha(p)-2\delta)_+\rangle \nonumber\\
&&+\langle (\beta_n'\alpha'\gamma_n(p)-4\delta)_+\rangle\nonumber\\
&&\leq\langle s\rangle+
  \langle (\beta_n\alpha(p)-2\delta)_+\rangle+\langle (\beta_n'\alpha'\gamma_n(p)-2\delta)_+\rangle\nonumber\\
&&\leq \langle q\rangle.\nonumber
\end{eqnarray}

$(\textbf{III.III})$, we suppose that   $(\beta_n\alpha (p)-2\delta)_+$ is Cuntz equivalent to a projection and   $\beta_n\alpha(q)$ is not Cuntz equivalent to a projection.
Choose a projection $p_0$ such that  $(\beta_n\alpha (p)-2\delta)_+$ is  Cuntz equivalent to $p_0.$
We may assume that $ (\beta_n\alpha (p)-2\delta)_+=p_0.$

By Theorem \ref {thm:2.1} (2),  there exists a non-zero positive element $d$  orthogonal to $\beta_n\alpha(q)$ such that
 $\langle(\beta_n\alpha (q)-4\delta)_++ d\rangle \leq\langle\beta_n\alpha(q)\rangle.$

   With  $G=\{\gamma_n(p), \gamma_n(q),~ \gamma_n(v_{i,j}): ~1\leq i\leq k,~ 1\leq j\leq k\},$  and any sufficiently small $\varepsilon''>0$,  $E=\gamma_n(1)A\gamma_n(1)$, since $E$ is  asymptotically tracially in $\Omega$, there exist
a ${\rm C^*}$-algebra $D$  in $\Omega$ and completely positive  contractive linear maps  $\alpha':E\to D$ and  $\beta_n': D\to E$, and $\gamma_n':E\to E\cap\beta_n'(D)^{\perp}$ such that

$(1)'$ the map $\alpha'$ is unital  completely positive   linear map, $\beta_n'(1_D)$ and $\gamma_n'(1_E)$ are all projections, $\beta_n'(1_D)+\gamma_n'(1_E)=1_E$ for all $n\in \mathbb{N}$,

$(2)'$ $\|x-\gamma_n'(x)-\beta_n'(\alpha'(x))\|<\varepsilon''$ for all $x\in G$ and for all $n\in {\mathbb{N}}$,

$(3)'$ $\alpha'$ is an $F$-$\varepsilon''$ approximate embedding,

$(4)'$ $\lim_{n\to \infty}\|\beta_n'(xy)-\beta_n'(x)\beta_n(y)\|=0$ and $\lim_{n\to \infty}\|\beta_n'(xy)\|=\|x\|$ for all $x,y\in D$, and

$(5)'$ $\gamma_n'\gamma_n(1)\lesssim d$ for all $n\in \mathbb{N}$.

Since
$$\|(\gamma_n\otimes id_{M_{k}}(v)){\rm diag}(\gamma_n(q)\otimes 1_l, ~0\otimes 1_{k-l})(\gamma_n\otimes id_{M_{k}}({v}^*))
-\gamma_n(p)\otimes 1_k\|
<\delta.$$

By $(1)'$, we have
\begin{eqnarray}
\label{Eq:eq1}
&&\|\alpha'\otimes id_{M_{k}}\gamma_n\otimes id_{M_{k}}(v){\rm diag}(\alpha'\gamma_n(q)\otimes 1_l, ~0\otimes 1_{k-l})\nonumber\\
&&\alpha'\otimes id_{M_{k}}\gamma_n\otimes id_{M_{k}}(v^*)-\alpha'\gamma_n(p)\otimes 1_k\|<2\delta.\nonumber
\end{eqnarray}
By Theorem \ref {thm:2.1} (1), we  have
$$k\langle(\alpha'\gamma_n(p)-4\delta)_+\rangle\leq l\langle (\alpha'\gamma_n(q)-\delta)_+
\rangle. $$
Since $D\in \Omega$, we have
$$\langle(\alpha'\gamma_n(p)-4\delta)_+\rangle\leq \langle (\alpha'\gamma_n(q)-\delta)_+
\rangle. $$
There exists $w\in D$ such that
$$\|w(\alpha'\gamma_n(q)-\delta)_+w^*-(\alpha'\gamma_n(p)-4\delta)_+\|<\delta.$$
By $(4')$, we have
$$\|\beta_n'(w)\beta_n'(\alpha'\gamma_n(q)-\delta)_+\beta_n(w^*)
-\beta_n((\alpha'\gamma_n(p)-4\delta)_+)\|<2\delta.$$
We  have   $$\langle (\beta_n'\alpha'\gamma_n(p)-8\delta)_+\rangle\leq \langle (\beta_n'\alpha'\gamma_n(q)-\delta)_+\rangle.$$

Therefore, we have
\begin{eqnarray}
\label{Eq:eq1}
&&\langle (p-\varepsilon)_+\rangle \nonumber\\
&&\leq\langle\gamma_n(p)-6\delta)_+\rangle+\langle (\beta_n\alpha(p)-2\delta)_+\rangle\nonumber\\
 &&\leq\langle(\gamma_n'\gamma_n(p)-4\delta)_+\rangle+\langle (\beta_n'\alpha'\gamma_n(p)\rangle\nonumber\\
  &&+\langle (\beta_n\alpha(p)-2\delta)_+\rangle\leq\langle\gamma_n'\gamma_n(1_A)\rangle+
  \langle (\beta_n\alpha(p)-2\delta)_+\rangle \nonumber\\
&&+\langle (\beta_n'\alpha'\gamma_n(p)-4\delta)_+\rangle\nonumber\\
&&\leq\langle s\rangle+
  \langle (\beta_n\alpha(p)-2\delta)_+\rangle+\langle (\beta_n'\alpha'\gamma_n(p)-2\delta)_+\rangle\nonumber\\
&&\leq \langle q\rangle.\nonumber
\end{eqnarray}
\end{proof}

\begin{theorem}\label{thm:3.2}
 Let $\mathcal{P}$ be a class of stably finite unital
${\rm C^*}$-algebras which have Winter's $n$-comparison.  Then $A$  has  Winter's  $n$-comparison  for  any  simple  unital ${\rm C^*}$-algebra $A\in{\rm ATA}\mathcal{P}.$
\end{theorem}

\begin{proof} As in the proof of Theorem 3.1, we may suppose that $\mathcal{P}$ is closed under Morita equivalence.

 Let  $a, b_0, b_1, \cdots, b_n\in {\rm M}_{\infty}(A)_+$. By  Theorem \ref {thm:2.1} (3),
  we need only  to show that $\langle(a-\varepsilon)_+\rangle\leq \langle b_0\rangle+\langle b_1\rangle + \cdots+\langle b_n\rangle$ for  any $\varepsilon>0$, if  $(k_i+1)\langle a\rangle \leq k_i\langle b_i\rangle,~ 0\leq i\leq n.$
  Note that $k_i$ can be chosen to be the same for all $b_i,$ as, with $k=(k_0+1)(k_1+1)\cdots (k_n+1)-1,$ one has  $(k+1)\langle a\rangle \leq k\langle b_i\rangle$ for all $0\leq i \leq n.$

We may assume that $a, b_0, b_1, \cdots, b_n\in {\rm M}_{k}(A)_+$ for some sufficiently large integer $k$. By Lemma \ref {lem:2.6}, and Lemma \ref{lem:2.7},  we may assume that
$a, b_0, b_1, \cdots, b_n\in A_+$.

    We  divide the proof into three cases.

   $(\textbf{I})$, let us  suppose that $b_i$ is  not Cuntz equivalent to a projection for some $0\leq i\leq n.$    We may assume that $b_0$ is a   purely positive element.
  By  Theorem \ref {thm:2.1} (2), for  $\delta>0$,  there is a  non-zero positive element $d$  orthogonal to $b_0$ such that
  $(b_0-\delta/2)_++d\lesssim b_0$. Hence, as the proof of Theorem \ref{thm:3.1}, by Theorem \ref {thm:2.1} (1), there exist  $v_k=(v_{i,j}^k), ~ 0\leq k\leq n,~1\leq i\leq k+1,~ 1\leq j\leq k+1$ such that
         $$\|v_i{\rm diag}(b_i\otimes 1_{k},~ 0)v_i^*-a\otimes1_{k+1}\|<\delta, $$
   where $0\leq i\leq n.$

Since $A\in {\rm ATA}\mathcal{P}$, for  $F=\{a,~ b_0, ~b_1,~ \cdots,~ b_n,~d,~ {v_{i,j}^k}: ~1\leq i\leq k+1,~ 1\leq j\leq k+1,~ 0\leq k \leq n\},$  and  $\varepsilon'>0$,
  there exist
a ${\rm C^*}$-algebra $B$  in $\mathcal{P}$ and completely positive  contractive linear maps  $\alpha:A\to B$ and  $\beta_n: B\to A$, and $\gamma_n:A\to A\cap\beta_n(B)^{\perp}$ such that

$(1)$ the map $\alpha$ is unital  completely positive   linear map, $\beta_n(1_B)$ and $\gamma_n(1_A)$ are all projections $\beta_n(1_B)+\gamma_n(1_A)=1_A$ for all $n\in \mathbb{N}$,

$(2)$ $\|x-\gamma_n(x)-\beta_n(\alpha(x))\|<\varepsilon'$ for all $x\in F$ and for all $n\in {\mathbb{N}}$,

$(3)$ $\alpha$ is an $F$-$\varepsilon'$ approximate embedding,

$(4)$ $\lim_{n\to \infty}\|\beta_n(xy)-\beta_n(x)\beta_n(y)\|=0$ and $\lim_{n\to \infty}\|\beta_n(xy)\|=\|x\|$ for all $x,y\in B$, and

$(5)$ $\gamma_n(1)\lesssim d$ for all $n\in \mathbb{N}$.

Since  $$\|v_i{\rm diag}(b_i\otimes 1_{k},~ 0)v_i^*-a\otimes1_{k+1}\|<\delta, $$
   for all $0\leq i\leq n.$

By $(3)$, we have
$$\|\alpha\otimes id_{M_{k+1}}(v_i){\rm diag}(\alpha(b_i)\otimes 1_k, ~0))
\alpha\otimes id_{M_{k+1}}({v_i}^*)-\alpha(a)\otimes 1_{k+1}\|
<\delta,$$
for all $0\leq i\leq n.$

By  Theorem \ref {thm:2.1} (1), we have
$$(k+1)\langle(\alpha(a)-\delta)_+\rangle\leq \langle\alpha(b_i)\rangle,$$
for all $0\leq i\leq n.$

Since $B\in \Omega,$ we have $$\langle(\alpha(a)-\delta)_+\rangle\leq \langle\alpha(b_0)\rangle+\langle\alpha(b_1)
\rangle+\cdots+\langle\alpha(b_n)\rangle.$$
Since $$\langle(\alpha(a)-\delta)_+\rangle\leq \langle\alpha(b_0)\rangle+\langle\alpha(b_1)
\rangle+\cdots+\langle\alpha(b_n)\rangle,$$

There exists $w\in M_{k+1}(D)$ such that
$$\|w {\rm diag}(\alpha(b_0), \alpha(b_1),\cdots, \alpha(b_n))w^*-(\alpha(a)-\delta)_+\|<\delta.$$

By $(4)$, we have
$$ \|\beta_n\otimes id_{M_{k+1}}(w) {\rm diag}(\beta_n\alpha(b_0), \beta_n\alpha(b_1),\cdots, \beta_n\alpha(b_n)) \beta_n\otimes id_{M_{k+1}}(w^*)$$$$-(\beta_n\alpha(a)-\delta)_+\|<2\delta.$$

By Theorem \ref {thm:2.1} (1), we have
$$\langle(\beta_n\alpha(a)-4\delta)_+\rangle\leq \langle\beta_n\alpha(b_0-\delta)_+\rangle+\langle\beta_n\alpha(b_1)
\rangle+\cdots+\langle\beta_n\alpha(b_n)\rangle,$$

Therefore,  we have
\begin{eqnarray}
\label{Eq:eq1}
&&\langle(a-\varepsilon)_+\rangle \nonumber\\
&&\leq\langle\gamma_n(a)-\delta)_+\rangle+\langle (\beta_n\alpha(a)-4\delta)_+\rangle\nonumber\\
&&\leq\langle\gamma_n(a)\rangle+\langle\beta_n\alpha(b_0-\delta)_+)\rangle+\langle
\beta_n\alpha(b_1)
\rangle\nonumber\\
&&+\cdots+\langle\beta_n\alpha(b_n)\rangle \leq\langle\gamma_n(1_A)\rangle+\langle\beta_n\alpha(b_0-\delta)_+\rangle\nonumber\\
&&+\langle\beta_n\alpha(b_1)
\rangle+\cdots+\langle\beta_n\alpha(b_n)\rangle\nonumber\\
&&\leq\langle d\rangle+\langle\beta_n\alpha(b_0-\delta)_+\rangle\nonumber\\
&&+\langle\beta_n\alpha(b_1)
\rangle+\cdots+\langle\beta_n\alpha(b_n)\rangle\nonumber\\
&&\leq\langle (b_0-\delta/2)_+\rangle
+\langle d\rangle+\langle b_1\rangle+\cdots +\langle b_n\rangle\leq\langle b_0\rangle+\langle b_1\rangle+\cdots +\langle b_n\rangle.\nonumber
\end{eqnarray}

$(\textbf{II})$,  let us  suppose that each $b_i ~(0\leq i\leq n)$ is Cuntz equivalent to a projection  and $a$ is not Cuntz equivalent to a projection.
Choose  projections $p_0,~ p_1,~\cdots,~ p_n$ such that $b_i$  is Cuntz equivalent to $p_i$ for all $0\leq i\leq n.$
We may assume that $b_i=p_i$ for all  $0\leq i\leq n.$

By Theorem \ref {thm:2.1} (2), let $\delta>0$ (with $\delta<\varepsilon$), there is a non-zero positive element $d$ orthogonal to $a$  such that
 $\langle(a-\delta)_++ d\rangle\leq\langle a\rangle$.
Since $(k+1)\langle a\rangle \leq k\langle p_i\rangle,$ we have $(k+1)\langle(a-\delta)_++ d\rangle \leq k\langle p_i\rangle$ for all $0\leq i\leq n.$

Hence, by Theorem \ref {thm:2.1} (1),  as earlier,  there exist  $v_k=(v_{i,j}^k) \in {\rm M}_{k+1}(A), $ $0\leq k\leq n,~1\leq i\leq k+1,~ 1\leq j\leq k+1$, such that
   $$\|v_i{\rm diag}(p_i\otimes 1_{k},  ~0)v_i^*-((a-\delta)_++ d) \otimes1_{k+1}\|<\delta,~ 0\leq i\leq n.$$

Since $A\in {\rm ATA}\mathcal{P},$   for $F=\{a,~ b_0,~ b_1,~\cdots,~ b_n, ~d,~ {v_{i,j}^k}: ~1\leq i\leq k+1,~ 1\leq j\leq k+1, ~0\leq k\leq n\},$  for any  $\varepsilon'>0$,
  there exist
a ${\rm C^*}$-algebra $B$  in $\mathcal{P}$ and completely positive  contractive linear maps  $\alpha:A\to B$ and  $\beta_n: B\to A$, and $\gamma_n:A\to A\cap\beta_n(B)^{\perp}$ such that

$(1)$ the map $\alpha$ is unital  completely positive   linear map, $\beta_n(1_B)$ and $\gamma_n(1_A)$ are all projections $\beta_n(1_B)+\gamma_n(1_A)=1_A$ for all $n\in \mathbb{N}$,

$(2)$ $\|x-\gamma_n(x)-\beta_n(\alpha(x))\|<\varepsilon'$ for all $x\in F$ and for all $n\in {\mathbb{N}}$.

$(3)$ $\alpha$ is an $F$-$\varepsilon'$ approximate embedding,

$(4)$ $\lim_{n\to \infty}\|\beta_n(xy)-\beta_n(x)\beta_n(y)\|=0$ and $\lim_{n\to \infty}\|\beta_n(xy)\|=\|x\|$ for all $x,y\in B$.

Since $\|v_i{\rm diag}(p_i\otimes 1_{k},  ~0)v_i^*-((a-\delta)_++ d) \otimes1_{k+1}\|<\delta,~ 0\leq i\leq n$,
 By $(2)$, $(3)$ and $(4)$, we have
 \begin{eqnarray}
\label{Eq:eq1}
&&\|(\gamma_n\otimes id_{M_{k+1}}(v_i)+\beta_n\otimes id_{M_{k+1}}(\alpha\otimes id_{M_{k}}(v_i)))({\rm diag}(\gamma_n(p_i)\otimes 1_k, ~0) \nonumber\\
&&+{\rm diag}(\beta_n\alpha(p_i)\otimes 1_k, ~0))(\gamma_n\otimes id_{M_{k+1}}({v_i}^*)+\beta_n\otimes id_{M_{k+1}}(\alpha\otimes id_{M_{k+1}}({v_i}^*)))\nonumber\\
&&-(\gamma_n((a-\delta)_++d)\otimes 1_{k+}, +\beta_n\alpha((a-\delta)_++d)\otimes 1_{k+1})\|<\delta.\nonumber
\end{eqnarray}

 By $(1)$, we have
 \begin{eqnarray}
\label{Eq:eq1}
&&\|(\gamma_n\otimes id_{M_{k+1}}(v_i))({\rm diag}(\gamma_n(p_i)\otimes 1_k, ~0))(\gamma_n\otimes id_{M_{k+1}}({v_i}^*))\nonumber\\
&&-\gamma_n((a-\delta)_++d)\otimes 1_{k+1}\|
<\delta.\nonumber
\end{eqnarray}

Since $\|v_i{\rm diag}(p_i\otimes 1_{k},  ~0)v_i^*-((a-\delta)_++ d) \otimes1_{k+1}\|<\delta,~ 0\leq i\leq n$,
by $(1)$ and $(3)$, we have
$$\|\alpha\otimes id_{M_{k+1}}(v)){\rm diag}(\alpha(p_i)\otimes 1_l, ~0)
\alpha\otimes id_{M_{k+1}}({v}^*)-\alpha((a-\delta)_++ d) \otimes1_{k+1}\|
<2\delta,$$
for $0\leq i\leq n$.

By Theorem \ref {thm:2.1} (1), we have
$$(k+1)\langle(\alpha((a-\delta)_++ d)-2\delta)_+\rangle\leq k\langle \alpha(p_i)\rangle$$
for $0\leq i\leq n$.

Since $B\in \Omega$, we have $$\langle(\alpha((a-\delta)_++ d)-2\delta)_+\rangle\leq
\langle \alpha(p_0)\rangle+\langle \alpha(p_1)\rangle+\cdots+\langle \alpha(p_n)\rangle.$$

Since $\langle(\alpha((a-\delta)_++ d)-2\delta)_+\rangle\leq
\langle \alpha(p_0)\rangle+\langle \alpha(p_1)\rangle+\cdots+\langle \alpha(p_n)\rangle,$
there exists $w\in M_{k+1}(B)$ such that
$$\|w {\rm diag}(\alpha(p_0), \alpha(p_1),\cdots, \alpha(p_n))w^*-(\alpha((a-\delta)_++ d)-2\delta)_+\|<\delta.$$

By $(4)$, we have
$$ \|\beta_n\otimes id_{M_{k+1}}(w) {\rm diag}(\beta_n\alpha(p_0), \beta_n\alpha(p_1),\cdots, \beta_n\alpha(p_n)) \beta_n\otimes id_{M_{k+1}}(w^*)$$$$-(\beta_n(\alpha((a-\delta)_++ d)-2\delta)_+)\|<2\delta.$$

By Theorem \ref {thm:2.1} (1), we have
$$\langle(\beta_n(\alpha((a-\delta)_++ d)-6\delta)_+)\rangle\leq \langle(\beta_n\alpha(p_0)-\delta)_+\rangle+\langle\beta_n\alpha(p_1)
\rangle+\cdots+\langle\beta_n\alpha(p_n)\rangle.$$

Since $(a-\delta)_+$  is orthogonal to $d$, by $(3)$ and
$(4)$, we may assume that $\beta_n\alpha((a-\delta)_+)$ is  orthogonal to $\beta_n\alpha(d)$.

With  $G=\{\gamma_n(a), \gamma_n(p_i),~ \gamma_n(v_{i,j}^k): ~1\leq i\leq k+1,~ 1\leq j\leq k+1, 0\leq k\leq n\},$  and any $\varepsilon''>0$,  $E=\gamma_n(1)A\gamma_n(1)$, since $E$ is  asymptotically tracially in $\mathcal{P}$, there exist
a ${\rm C^*}$-algebra $D$  in $\mathcal{P}$ and completely positive  contractive linear maps  $\alpha':E\to D$ and  $\beta_n': D\to E$, and $\gamma_n':E\to E\cap\beta_n'(D)^{\perp}$ such that

$(1')$ the map $\alpha'$ is unital  completely positive   linear map,  $\beta_n'(1_D)$ and $\gamma_n'(1_A)$ are all projections, and  $\beta_n'(1_D)+\gamma_n'(1_A)=1_A$ for all $n\in \mathbb{N}$,

$(2')$ $\|x-\gamma_n'(x)-\beta_n'(\alpha'(x))\|<\varepsilon''$ for all $x\in G$ and for all $n\in {\mathbb{N}}$,

$(3')$ $\alpha'$ is an $F$-$\varepsilon''$ approximate embedding,

$(4')$ $\lim_{n\to \infty}\|\beta_n'(xy)-\beta_n'(x)\beta_n(y)\|=0$ and $\lim_{n\to \infty}\|\beta_n'(xy)\|=\|x\|$ for all $x,y\in D$, and

$(5')$ $\gamma_n'\gamma_n(1)\lesssim \beta_n\alpha(d)$ for all $n\in \mathbb{N}$.

Since
$$\|(\gamma_n\otimes id_{M_{k+1}}(v_i)){\rm diag}(\gamma_n(p_i)\otimes 1_k, ~0)(\gamma_n\otimes id_{M_{k}}({v}^*))-\gamma_n(a)\otimes 1_{k+1}\|<\delta,$$
for $0\leq i\leq n$,
by $(1)'$, we have

$$\|\alpha'\otimes id_{M_{k+1}}\gamma_n\otimes id_{M_{k+1}}(v_i){\rm diag}(\alpha'\gamma_n(p_i)\otimes 1_k, ~0)$$$$
\alpha'\otimes id_{M_{k+1}}(\gamma_n\otimes id_{M_{k+1}}(v_i^*))-\alpha'\gamma_n(a)\otimes 1_{k+1}\|<2\delta.$$

By Theorem \ref {thm:2.1} (1), we  have
$$(k+1)\langle(\alpha'\gamma_n(a)-4\delta)_+\rangle\leq k\langle \alpha'\gamma_n(p_i)\rangle. $$
Since $D\in \Omega$, we have
$$\langle(\alpha'\gamma_n(a)-4\delta)_+\rangle\leq \langle \alpha'\gamma_n(p_0)
\rangle+\alpha'\gamma_n(p_1)
\rangle+\cdots+\alpha'\gamma_n(p_n)
\rangle. $$

Since $\langle(\alpha'\gamma_n(a)-4\delta)_+\rangle\leq \langle (\alpha'\gamma_n(p_0)-\delta)_+
\rangle+(\alpha'\gamma_n(p_1)-\delta)_+
\rangle+\cdots+(\alpha'\gamma_n(p_n)-\delta)_+
\rangle, $
there exists $w\in M_{k+1}(D)$ such that
$$\|w {\rm diag}(\alpha'\gamma_n(p_0), \alpha'\gamma_n(p_1),\cdots, \alpha'\gamma_n(p_n))w^*-(\alpha'\gamma_n((a-\delta)_++ d)-2\delta)_+\|<\delta.$$

By $(4)$, we have
$$ \|\beta_n'\otimes id_{M_{k+1}}(w) {\rm diag}(\beta_n'\alpha'\gamma_n(p_0), \beta_n'\alpha'\beta_n\alpha(p_1),\cdots, \beta_n'\alpha'\beta_n\alpha(p_n)) \beta_n\otimes id_{M_{k+1}}(w^*)$$$$-(\beta_n'\alpha'\beta_n(\alpha((a-\delta-2\delta)_+)\|<2\delta.$$

By Theorem \ref {thm:2.1} (1), we have
$$\|\langle(\beta_n'(\alpha'((a-\delta)_+-2\delta)_+)\rangle\leq \langle\beta_n'\alpha'\gamma_n(p_0)\rangle+\langle\beta_n'\alpha'\gamma_n(p_1)
\rangle+\cdots+\langle\beta_n'\alpha'\gamma_n(p_n)\rangle.$$

 Therefore, we have
 \begin{eqnarray}
\label{Eq:eq1}
&&\langle(a-\varepsilon)_+\rangle \nonumber\\
&&\leq\langle(\gamma_n(a)-10\delta)_+\rangle+\langle (\beta_n\alpha(a-\delta)_+)-8\delta)_+\rangle\nonumber\\
 &&\leq\langle(\gamma_n'\gamma_n(a)-2\delta)_+\rangle+\langle (\beta_n'\alpha'\gamma_n((a-\delta)_+)-8\delta)_+\rangle\nonumber\\
  &&+\langle (\beta_n\alpha((a-\delta)_+)-8\delta)_+\rangle\leq\langle\gamma_n'(1_A)\rangle \nonumber\\
&&+ \langle (\beta_n\alpha((a-\delta)_+)-8\delta)_+\rangle+\langle (\beta_n'\alpha'\gamma_n((a-\delta)_+)-2\delta)_+\rangle\nonumber\\
&&\leq\langle\beta_n\alpha(d)\rangle+
  \langle (\beta_n\alpha((a-\delta)_+)-2\delta)_+\rangle
  +\langle (\beta_n'\alpha'\gamma_n((a-\delta)_+)-2\delta)_+\rangle\nonumber\\
&&\leq\langle(\beta_n\alpha(a)-\delta)_+\rangle+\langle(\beta_n\alpha(p_1)
\rangle+\cdots+\langle(\beta_n\alpha(p_n)\rangle \nonumber\\
&&+\langle (\beta_n'\alpha'\gamma_n(p_0)-\delta)_+\rangle+
\langle (\beta_n'\alpha'\gamma_n(p_1)-\delta)_+\rangle+\cdots+
\langle (\beta_n'\alpha'\gamma_n(p_n)-\delta)_+\rangle\nonumber\\
&&\leq\langle(\beta_n\alpha(p_0)-\delta)_+\rangle+\langle(\beta_n\alpha(p_1)-\delta)_+
\rangle+\cdots+\langle(\beta_n\alpha(p_n)-\delta)_+\rangle \nonumber\\
&&\leq\langle p_0\rangle+\langle p_1\rangle+\cdots +\langle p_n\rangle.\nonumber
\end{eqnarray}

  $(\textbf{III})$, we suppose that both $a$ and $b_i$ ($0\leq i\leq n$) are  all  Cuntz equivalent to projections.

   Choose  projections $p,~ q_i$ such that $a$ is Cuntz equivalent to $p$ and  $b_i$  is Cuntz equivalent to $q_i.$
We may assume that $a=p,~ b_i=q_i.$

Since $(k+1)\langle p\rangle \leq k\langle q_i\rangle$ for all $0\leq i \leq n,$ so for any $\delta>0$,
by Theorem \ref {thm:2.1} (1), there exist $v_k=(v_{i,j}^k), ~ 0\leq k\leq n,~1\leq i\leq k+1,~ 1\leq j\leq k+1$ such that
         $$\|v_i{\rm diag}(q_i\otimes 1_{k},~ 0)v_i^*-p \otimes1_{k+1}\|<\delta.$$

Since $A\in {\rm ATA}\mathcal{P}$, for $F=\{p,~ q_k, (v_{i,j}^k), ~ 0\leq k\leq n,~1\leq i\leq k+1,~ 1\leq j\leq k+1\},$ for  any $\varepsilon'>0,$
  there exist a
a ${\rm C^*}$-algebra $B$  in $\mathcal{P}$ and completely positive  contractive linear maps  $\alpha:A\to B$ and  $\beta_n: B\to A$, and $\gamma_n:A\to A\cap\beta_n(B)^{\perp}$ such that

$(1)$ the map $\alpha$ is unital  completely positive   linear map, $\beta_n(1_B)$ and $\gamma_n(1_A)$ are all projections $\beta_n(1_B)+\gamma_n(1_A)=1_A$ for all $n\in \mathbb{N}$,

$(2)$ $\|x-\gamma_n(x)-\beta_n(\alpha(x))\|<\varepsilon'$ for all $x\in F$ and for all $n\in {\mathbb{N}}$.

$(3)$ $\alpha$ is an $F$-$\varepsilon'$ approximate embedding,

$(4)$ $\lim_{n\to \infty}\|\beta_n(xy)-\beta_n(x)\beta_n(y)\|=0$ and $\lim_{n\to \infty}\|\beta_n(xy)\|=\|x\|$ for all $x,y\in B$.

Since $\|v_i{\rm diag}(q_i\otimes 1_{k},~ 0)v_i^*-p \otimes1_{k+1}\|<\delta$ for $0\leq i\leq n$,
 by $(2)$, we have
 \begin{eqnarray}
\label{Eq:eq1}
&&\|(\gamma_n\otimes id_{M_{k+1}}(v_i)+\beta_n\otimes id_{M_{k+1}}(\alpha\otimes id_{M_{k}}(v_i)))({\rm diag}(\gamma_n(q_i)\otimes 1_k, ~0) \nonumber\\
&&+{\rm diag}(\beta_n\alpha(q_i)\otimes 1_k, ~0))(\gamma_n\otimes id_{M_{k+1}}({v_i}^*)+\beta_n\otimes id_{M_{k+1}}(\alpha\otimes id_{M_{k+1}}({v_i}^*)))\nonumber\\
&&-(\gamma_n(p)\otimes 1_{k+1} +\beta_n\alpha(p)\otimes 1_{k+1})\|\nonumber\\
&&<2\delta.\nonumber
\end{eqnarray}

 By $(1)$, we have
 \begin{eqnarray}
\label{Eq:eq1}
&&\|(\gamma_n\otimes id_{M_{k+1}}(v_i))({\rm diag}(\gamma_n(q_i)\otimes 1_k, ~0))(\gamma_n\otimes id_{M_{k+1}}({v_i}^*))\nonumber\\
&&-\gamma_n(p)\otimes 1_{k+1}\|
<2\delta.\nonumber
\end{eqnarray}

Since $ \|v_i{\rm diag} (q_i\otimes 1_{k}, ~0) v_i^*-p\otimes1_{k+1}\|<\delta$ for $0\leq i\leq n$,
by $(1)$ and $(3)$, we have
$$\|\alpha\otimes id_{M_{k+1}}(v_i){\rm diag}(\alpha(q_i)\otimes 1_k, ~0)
\alpha\otimes id_{M_{k+1}}({v_i}^*)-\alpha(p)\otimes 1_{k+1}\|
<\delta.$$

By Theorem  \ref {thm:2.1} (1), we have
$$(k+1)\langle(\alpha(p)-\delta)_+\rangle\leq k\langle\alpha(q_i)\rangle.$$
Since $B\in \Omega$, we have
$$\langle(\alpha(p)-\delta)_+\rangle\leq \langle\alpha(q_0)\rangle+ \langle\alpha(q_1)\rangle+\cdots\ \langle\alpha(q_n)\rangle.$$
Since$\langle(\alpha(p)-\delta)_+\rangle\leq \langle\alpha(q_0)\rangle+ \langle\alpha(q_1)\rangle+\cdots\ \langle\alpha(q_n)\rangle,$
there exists $w\in M_{k+1}(D)$ such that
$$\|w {\rm diag}(\alpha(p_0), \alpha(p_1),\cdots, \alpha(p_n))w^*-(\alpha((a-\delta)_+)\|<\delta.$$

By $(4)$, we have
$$ \|\beta_n\otimes id_{M_{k+1}}(w) {\rm diag}(\beta_n\alpha(p_0), \beta_n\alpha(p_1),\cdots, \beta_n\alpha(p_n)) \beta_n\otimes id_{M_{k+1}}(w^*)$$$$-(\beta_n(\alpha((a-\delta)_+)-2\delta)_+)\|<2\delta.$$

By Theorem \ref {thm:2.1} (1), we have
$$\langle(\beta_n(\alpha((a-\delta)_+)-2\delta)_+)\rangle\leq \langle\beta_n\alpha(p_0)\rangle+\langle\beta_n\alpha(p_1)
\rangle+\cdots+\langle\beta_n\alpha(p_n)\rangle.$$

 $(\textbf{III.I})$,
if $(\beta_n\alpha (p)-2\delta)_+$ and $\beta_n\alpha(q_i)$  are not Cuntz equivalent to a pure positive element, then there exist projections $p',q'_i$ such that $(\beta_n\alpha (p)-2\delta)_+\sim p'$ and $\beta_n\alpha(q_i)\sim q'_i$, and we suppose that $p'\sim  q_0'+q_1'+\cdots+q_n'$,
 then we have $(k+1)\langle p'\rangle=(k+1)\langle q_0'\rangle+(k+1)\langle q_1'\rangle+\cdots +(k+1)\langle q_n'\rangle
 \leq k\langle q_0'\rangle+k\langle q_1'\rangle+\cdots +k\langle q_n'\rangle,$
 and this
contradicts  the stable finiteness of $A$ (since ${\rm C^*}$-algebras in $\mathcal{P}$ are stably finite (cf.~\ proposition 4.2 in  \cite{FL}).
So there exist a nonzero projection $s\in A$, orthogonal to $p'$  such that
 $$(\beta_n\alpha (p)-2\delta)_++s\lesssim q_0'+q_1'+\cdots+q_n'.$$

With   $G=\{\gamma_n(a), \gamma_n(p_i),~ \gamma_n(v_{i,j}^k): ~1\leq i\leq k+1,~ 1\leq j\leq k+1, 0\leq k\leq n\},$  and any $\varepsilon''>0$, $E=\gamma_n(1)A\gamma_n(1)$, since $E$ is  asymptotically tracially in $\mathcal{P}$,  there exist
a ${\rm C^*}$-algebra $D$  in $\Omega$ and completely positive  contractive linear maps  $\alpha':E\to D$ and  $\beta_n': D\to E$, and $\gamma_n':E\to E\cap\beta_n'(D)^{\perp}$ such that

$(1)'$ the map $\alpha'$ is unital  completely positive   linear map, $\beta_n'(1_D)$ and $\gamma_n'(1_E)$ are all projections, $\beta_n'(1_D)+\gamma_n'(1_E)=1_E$ for all $n\in \mathbb{N}$,

$(2)'$ $\|x-\gamma_n'(x)-\beta_n'(\alpha'(x))\|<\varepsilon''$ for all $x\in G$ and for all $n\in {\mathbb{N}}$,

$(3)'$ $\alpha'$ is an $F$-$\varepsilon''$ approximate embedding,

$(4)'$ $\lim_{n\to \infty}\|\beta_n'(xy)-\beta_n'(x)\beta_n(y)\|=0$ and $\lim_{n\to \infty}\|\beta_n'(xy)\|=\|x\|$ for all $x,y\in D$, and

$(5)'$ $\gamma_n'\gamma_n(1)\lesssim s$ for all $n\in \mathbb{N}$.

Since
$$\|(\gamma_n\otimes id_{M_{k+1}}(v_i)){\rm diag}(\gamma_n(q_i)\otimes 1_k, ~0)(\gamma_n\otimes id_{M_{k}}({v}^*))-\gamma_n(p)\otimes 1_{k+1}\|<2\delta,$$
for $0\leq i\leq n$,
by $(1)'$, we have

$$\|\alpha'\otimes id_{M_{k+1}}\gamma_n\otimes id_{M_{k+1}}(v_i){\rm diag}(\alpha'\gamma_n(q_i)\otimes 1_k, ~0)$$$$
\alpha'\otimes id_{M_{k+1}}(\gamma_n\otimes id_{M_{k+1}}(v_i^*))-\alpha'\gamma_n(p)\otimes 1_{k+1}\|<3\delta.$$

By Theorem \ref {thm:2.1} (1), we  have
$$(k+1)\langle(\alpha'\gamma_n(p)-(n+6)\delta)_+\rangle\leq k\langle (\alpha'\gamma_n(q_i)-\delta)_+
\rangle. $$
Since $D\in \Omega$, we have
$$\langle(\alpha'\gamma_n(p)-(n+6)\delta)_+\rangle\leq \langle (\alpha'\gamma_n(q_0)-\delta)_+
\rangle+\langle(\alpha'\gamma_n(q_1)-\delta)_+
\rangle+\cdots+\langle(\alpha'\gamma_n(q_n)-\delta)_+
\rangle. $$

Since $\langle(\alpha'\gamma_n(p)-(n+6)\delta)_+\rangle\leq \langle (\alpha'\gamma_n(q_0)-\delta)_+
\rangle+\langle(\alpha'\gamma_n(q_1)-\delta)_+
\rangle+\cdots+\langle(\alpha'\gamma_n(q_n)-\delta)_+
\rangle, $
there exists $w\in M_{k+1}(D)$ such that
$$\|w {\rm diag}((\alpha'\gamma_n(q_0)-\delta)_+, (\alpha'\gamma_n(q_1)-\delta)_+,\cdots, (\alpha'\gamma_n(q_n)-\delta)_+)w^*-(\alpha'\gamma_n(p)-(n+3)\delta)_+\|<\delta.$$

By $(4)$, we have
$$ \|\beta_n'\otimes id_{M_{k+1}}(w) {\rm diag}((\beta_n'\alpha'\gamma_n(q_0)-\delta)_+, (\beta_n'\alpha'\gamma_n(q_1)-\delta)_+,\cdots, (\beta_n'\alpha'\gamma_n(q_n)-\delta)_+)$$$$ \beta_n'\otimes id_{M_{k+1}}(w^*)-(\beta_n'(\alpha'\gamma_n(p)-(n+3)\delta)_+\|<2\delta.$$

By Theorem \ref {thm:2.1} (1), we have
$$\langle\beta_n'(\alpha'\gamma_n(p)-(n+10)\delta)_+\rangle\leq \langle(\beta_n'\alpha'\gamma_n(q_0)-\delta)_+\rangle+\langle(\beta_n'\alpha'\gamma_n(q_1)-\delta)_+
\rangle+\cdots+\langle(\beta_n'\alpha'\gamma_n(q_n)-\delta)_+\rangle.$$

Therefore, we have
 \begin{eqnarray}
\label{Eq:eq1}
&&\langle p\rangle \leq\langle(\gamma_n(p)-(n+10)\delta)_+)\rangle+\langle (\beta_n\alpha(p)-(n+10)\delta)_+\rangle\nonumber\\
 &&\leq\langle\gamma_n'\gamma_n(p)\rangle+\langle (\beta_n'\alpha'\gamma_n(p)-(n+12)\delta)_+\rangle+\langle (\beta_n\alpha(p)-3\delta)_+\rangle\leq\langle\gamma_n'\gamma_n(1_A)\rangle \nonumber\\
&&+ \langle (\beta_n\alpha((a-\delta)_+)-3\delta)_+\rangle+\langle (\beta_n'\alpha'\gamma_n((a-(n+8)\delta)_+)-2\delta)_+\rangle\nonumber\\
&&\leq\langle s \rangle+
  \langle (\beta_n\alpha(p)-3\delta)_+\rangle
  +\langle (\beta_n'\alpha'\gamma_n(p)-(n+10)\delta)_+\rangle\nonumber\\
&&\leq\langle(\beta_n\alpha(q_0)-\delta)_+\rangle+\langle(\beta_n\alpha(q_1)-\delta)_+
\rangle+\cdots+\langle(\beta_n\alpha(q_n)-\delta)_+\rangle \nonumber\\
&&+\langle (\beta_n'\alpha'\gamma_n(q_0)-\delta)_+\rangle+
\langle (\beta_n'\alpha'\gamma_n(q_1)-\delta)_+\rangle+\cdots+
\langle (\beta_n'\alpha'\gamma_n(q_n)-\delta)_+\rangle\nonumber\\
&&\leq\langle(\beta_n\alpha(q_0)-\delta)_+\rangle+\langle(\beta_n\alpha(q_1)-\delta)_+
\rangle+\cdots+\langle(\beta_n\alpha(q_n)-\delta)_+\rangle \nonumber\\
&&\leq\langle q_0\rangle+\langle q_1\rangle+\cdots +\langle q_n\rangle.\nonumber
\end{eqnarray}

$(\textbf{III.II})$, We suppose that $(\beta_n\alpha (p)-2\delta)_+$ is a purely positive element, then
By Theorem \ref {thm:2.1} (2),   there is a non-zero positive element $d$ orthogonal to $(\beta_n\alpha (p)-2\delta)_+$  such that
  $$(\beta_n\alpha (p)-3\delta)_++d\lesssim (\beta_n\alpha (p)-2\delta)_+.$$

With  $G=\{\gamma_n(a), \gamma_n(p_i),~ \gamma_n(v_{i,j}^k): ~1\leq i\leq k+1,~ 1\leq j\leq k+1, 0\leq k\leq n\},$  and $\varepsilon''>0$, $E=\gamma_n(1)A\gamma_n(1)$, since $E$ is  asymptotically tracially in $\mathcal{P}$, there exist
a ${\rm C^*}$-algebra $D$  in $\mathcal{P}$ and completely positive  contractive linear maps  $\alpha':E\to D$ and  $\beta_n': D\to E$, and $\gamma_n':E\to E\cap\beta_n'(D)^{\perp}$ such that

$(1)'$ the map $\alpha''$ is unital  completely positive   linear map, $\beta_n'(1_D)$ and $\gamma_n'(1_A)$ are all projections, $\beta_n'(1_D)+\gamma_n'(1_A)=1_A$ for all $n\in \mathbb{N}$,

$(2)'$ $\|x-\gamma_n'(x)-\beta_n'(\alpha'(x))\|<\varepsilon''$ for all $x\in G$ and for all $n\in {\mathbb{N}}$,

$(3)'$ $\alpha'$ is an $F$-$\varepsilon''$ approximate embedding,

$(4)'$ $\lim_{n\to \infty}\|\beta_n'(xy)-\beta_n'(x)\beta_n(y)\|=0$ and $\lim_{n\to \infty}\|\beta_n'(xy)\|=\|x\|$ for all $x,y\in D$, and

$(5)'$ $\gamma_n'\gamma_n(1)\lesssim d$ for all $n\in \mathbb{N}$.

Since
$$\|(\gamma_n\otimes id_{M_{k+1}}(v_i)){\rm diag}(\gamma_n(q_i)\otimes 1_k, ~0)(\gamma_n\otimes id_{M_{k}}({v}^*))-\gamma_n(p)\otimes 1_{k+1}\|<2\delta,$$
for $0\leq i\leq n$,
by $(1)'$, we have

$$\|\alpha'\otimes id_{M_{k+1}}\gamma_n\otimes id_{M_{k+1}}(v_i){\rm diag}(\alpha'\gamma_n(q_i)\otimes 1_k, ~0)$$$$
\alpha'\otimes id_{M_{k+1}}(\gamma_n\otimes id_{M_{k+1}}(v_i^*))-\alpha'\gamma_n(p)\otimes 1_{k+1}\|<3\delta.$$

By Theorem \ref {thm:2.1} (1), we  have
$$(k+1)\langle(\alpha'\gamma_n(p)-(n+6)\delta)_+\rangle\leq k\langle (\alpha'\gamma_n(q_i)-\delta)_+
\rangle. $$
Since $D\in \Omega$, we have
$$\langle(\alpha'\gamma_n(p)-(n+6)\delta)_+\rangle\leq \langle (\alpha'\gamma_n(q_0)-\delta)_+
\rangle+\langle(\alpha'\gamma_n(q_1)-\delta)_+
\rangle+\cdots+\langle(\alpha'\gamma_n(q_n)-\delta)_+
\rangle. $$

Since $\langle(\alpha'\gamma_n(p)-(n+6)\delta)_+\rangle\leq \langle (\alpha'\gamma_n(q_0)-\delta)_+
\rangle+\langle(\alpha'\gamma_n(q_1)-\delta)_+
\rangle+\cdots+\langle(\alpha'\gamma_n(q_n)-\delta)_+
\rangle, $
there exists $w\in M_{k+1}(D)$ such that
$$\|w {\rm diag}((\alpha'\gamma_n(q_0)-\delta)_+, (\alpha'\gamma_n(q_1)-\delta)_+,\cdots, (\alpha'\gamma_n(q_n)-\delta)_+)w^*-(\alpha'\gamma_n(p)-(n+6)\delta)_+\|<\delta.$$

By $(4)$, we have
$$ \|\beta_n'\otimes id_{M_{k+1}}(w) {\rm diag}((\beta_n'\alpha'\gamma_n(q_0)-\delta)_+, (\beta_n'\alpha'\gamma_n(q_1)-\delta)_+,\cdots, (\beta_n'\alpha'\gamma_n(q_n)-\delta)_+) $$$$\beta_n'\otimes id_{M_{k+1}}(w^*)-(\beta_n'(\alpha'\gamma_n(p)-(n+6)\delta)_+\|<2\delta.$$

By Theorem \ref {thm:2.1} (1), we have
$$\langle\beta_n'(\alpha'\gamma_n(p)-(n+10)\delta)_+\rangle\leq \langle(\beta_n'\alpha'\gamma_n(q_0)-\delta)_+\rangle+\langle(\beta_n'\alpha'\gamma_n(q_1)-\delta)_+
\rangle+\cdots+\langle(\beta_n'\alpha'\gamma_n(q_n)-\delta)_+\rangle.$$

Therefore, we have
 \begin{eqnarray}
\label{Eq:eq1}
&&\langle p\rangle \leq\langle(\gamma_n(p)-(n+10)\delta)_+)\rangle+\langle (\beta_n\alpha(p)-(n+10)\delta)_+\rangle\nonumber\\
 &&\leq\langle\gamma_n'\gamma_n(p)\rangle+\langle (\beta_n'\alpha'\gamma_n(p)-(n+12)\delta)_+\rangle+\langle (\beta_n\alpha(p)-3\delta)_+\rangle\leq\langle\gamma_n'\gamma_n(1_A)\rangle \nonumber\\
&&+ \langle (\beta_n\alpha((a-\delta)_+)-3\delta)_+\rangle+\langle (\beta_n'\alpha'\gamma_n((a-(n+8)\delta)_+)-2\delta)_+\rangle\nonumber\\
&&\leq\langle d \rangle+
  \langle (\beta_n\alpha(p)-3\delta)_+\rangle
  +\langle (\beta_n'\alpha'\gamma_n(p)-(n+10)\delta)_+\rangle\nonumber\\
&&\leq\langle(\beta_n\alpha(q_0)-\delta)_+\rangle+\langle(\beta_n\alpha(q_1)-\delta)_+
\rangle+\cdots+\langle(\beta_n\alpha(q_n)-\delta)_+\rangle \nonumber\\
&&+\langle (\beta_n'\alpha'\gamma_n(q_0)-\delta)_+\rangle+
\langle (\beta_n'\alpha'\gamma_n(q_1)-\delta)_+\rangle+\cdots+
\langle (\beta_n'\alpha'\gamma_n(q_n)-\delta)_+\rangle\nonumber\\
&&\leq\langle(\beta_n\alpha(q_0)-\delta)_+\rangle+\langle(\beta_n\alpha(q_1)-\delta)_+
\rangle+\cdots+\langle(\beta_n\alpha(q_n)-\delta)_+\rangle \nonumber\\
&&\leq\langle q_0\rangle+\langle q_1\rangle+\cdots +\langle q_n\rangle.\nonumber
\end{eqnarray}

$(\textbf{III.III})$, we suppose that  $(\beta_n\alpha (p)-2\delta)_+$ is Cuntz equivalent to a projection and there exist some $q_i$  such that $\beta_n\alpha(q_i)$ is not Cuntz equivalent to a projection.
Choose a projection $p_0$ such that$(\beta_n\alpha (p)-2\delta)_+$ is  Cuntz equivalent to $p_0.$
We may assume that $(\beta_n\alpha (p)-2\delta)_+=p_0.$

We may also assume that $q_0$ such that $ \beta_n\alpha(q_0)$ is a pure positive. By Theorem \ref {thm:2.1} (2),
   there exists a non-zero positive element $d$  orthogonal to $ \beta_n\alpha(q_0)$  such that
 $\langle(\beta_n\alpha (q_0)-\delta)_++ d\rangle \leq\langle \beta_n\alpha (q_0)\rangle.$

   With   $G=\{\gamma_n(a), \gamma_n(p_i),~ \gamma_n(v_{i,j}^k): ~1\leq i\leq k+1,~ 1\leq j\leq k+1, 0\leq k\leq n\},$  and any sufficiently small $\varepsilon''>0$, $E=\gamma_n(1)A\gamma_n(1)$, since $E$ is  asymptotically tracially in $\mathcal{P}$, there exist
a ${\rm C^*}$-algebra $D$  in $\mathcal{P}$ and completely positive  contractive linear maps  $\alpha':E\to D$ and  $\beta_n': D\to E$, and $\gamma_n':E\to E\cap\beta_n'(D)^{\perp}$ such that

$(1')$ the map $\alpha'$ is unital  completely positive   linear map, $\beta_n'(1_D)$ and $\gamma_n'(1_A)$ are all projections, $\beta_n'(1_D)+\gamma_n'(1_A)=1_A$ for all $n\in \mathbb{N}$,

$(2')$ $\|x-\gamma_n'(x)-\beta_n'(\alpha'(x))\|<\varepsilon''$ for all $x\in G$ and for all $n\in {\mathbb{N}}$,

$(3')$ $\alpha'$ is an $F$-$\varepsilon''$ approximate embedding,

$(4')$ $\lim_{n\to \infty}\|\beta_n'(xy)-\beta_n'(x)\beta_n(y)\|=0$ and $\lim_{n\to \infty}\|\beta_n'(xy)\|=\|x\|$ for all $x,y\in D$, and

$(5')$ $\gamma_n'(1)\gamma_n\lesssim d$ for all $n\in \mathbb{N}$.

Since
$$\|(\gamma_n\otimes id_{M_{k+1}}(v_i)){\rm diag}(\gamma_n(q_i)\otimes 1_k, ~0)(\gamma_n\otimes id_{M_{k}}({v}^*))-\gamma_n(p)\otimes 1_{k+1}\|<2\delta,$$
for $0\leq i\leq n$,
by $(1)'$, we have

$$\|\alpha'\otimes id_{M_{k+1}}\gamma_n\otimes id_{M_{k+1}}(v_i){\rm diag}(\alpha'\gamma_n(q_i)\otimes 1_k, ~0)$$$$
\alpha'\otimes id_{M_{k+1}}(\gamma_n\otimes id_{M_{k+1}}(v_i^*))-\alpha'\gamma_n(p)\otimes 1_{k+1}\|<3\delta.$$

By Theorem \ref {thm:2.1} (1), we  have
$$(k+1)\langle(\alpha'\gamma_n(p)-(n+6)\delta)_+\rangle\leq k\langle (\alpha'\gamma_n(q_i)-\delta)_+
\rangle. $$
Since $D\in \Omega$, we have
$$\langle(\alpha'\gamma_n(p)-(n+6)\delta)_+\rangle\leq \langle (\alpha'\gamma_n(q_0)-\delta)_+
\rangle+\langle(\alpha'\gamma_n(q_1)-\delta)_+
\rangle+\cdots+\langle(\alpha'\gamma_n(q_n)-\delta)_+
\rangle. $$

Since $\langle(\alpha'\gamma_n(p)-(n+6)\delta)_+\rangle\leq \langle (\alpha'\gamma_n(q_0)-\delta)_+
\rangle+\langle(\alpha'\gamma_n(q_1)-\delta)_+
\rangle+\cdots+\langle(\alpha'\gamma_n(q_n)-\delta)_+
\rangle, $
there exists $w\in M_{k+1}(D)$ such that
$$\|w {\rm diag}((\alpha'\gamma_n(q_0)-\delta)_+, (\alpha'\gamma_n(q_1)-\delta)_+,\cdots, (\alpha'\gamma_n(q_n)-\delta)_+)w^*-(\alpha'\gamma_n(p)-(n+3)\delta)_+\|<\delta.$$

By $(4)$, we have
$$ \|\beta_n'\otimes id_{M_{k+1}}(w) {\rm diag}((\beta_n'\alpha'\gamma_n(q_0)-\delta)_+, (\beta_n'\alpha'\gamma_n(q_1)-\delta)_+,\cdots, (\beta_n'\alpha'\gamma_n(q_n)-\delta)_+) $$$$\beta_n'\otimes id_{M_{k+1}}(w^*)-(\beta_n'(\alpha'\gamma_n(p)-(n+3)\delta)_+\|<2\delta.$$

By Theorem \ref {thm:2.1} (1), we have
$$\langle\beta_n'(\alpha'\gamma_n(p)-(n+10)\delta)_+\rangle\leq \langle(\beta_n'\alpha'\gamma_n(q_0)-\delta)_+\rangle+\langle(\beta_n'\alpha'\gamma_n(q_1)-\delta)_+
\rangle+\cdots+\langle(\beta_n'\alpha'\gamma_n(q_n)-\delta)_+\rangle.$$

Therefore, we have
 \begin{eqnarray}
\label{Eq:eq1}
&&\langle p\rangle \leq\langle(\gamma_n(p)-(n+10)\delta)_+)\rangle+\langle (\beta_n\alpha(p)-(n+10)\delta)_+\rangle\nonumber\\
 &&\leq\langle\gamma_n'\gamma_n(p)\rangle+\langle (\beta_n'\alpha'\gamma_n(p)-(n+12)\delta)_+\rangle+\langle (\beta_n\alpha(p)-3\delta)_+\rangle\leq\langle\gamma_n'\gamma_n(1_A)\rangle \nonumber\\
&&+ \langle (\beta_n\alpha((a-\delta)_+)-3\delta)_+\rangle+\langle (\beta_n'\alpha'\gamma_n((a-(n+8)\delta)_+)-2\delta)_+\rangle\nonumber\\
&&\leq\langle d \rangle+
  \langle (\beta_n\alpha(p)-3\delta)_+\rangle
  +\langle (\beta_n'\alpha'\gamma_n(p)-(n+10)\delta)_+\rangle\nonumber\\
&&\leq\langle(\beta_n\alpha(q_0)-\delta)_+\rangle+\langle(\beta_n\alpha(q_1)-\delta)_+
\rangle+\cdots+\langle(\beta_n\alpha(q_n)-\delta)_+\rangle \nonumber\\
&&+\langle (\beta_n'\alpha'\gamma_n(q_0)-\delta)_+\rangle+
\langle (\beta_n'\alpha'\gamma_n(q_1)-\delta)_+\rangle+\cdots+
\langle (\beta_n'\alpha'\gamma_n(q_n)-\delta)_+\rangle\nonumber\\
&&\leq\langle(\beta_n\alpha(q_0)-\delta)_+\rangle+\langle(\beta_n\alpha(q_1)-\delta)_+
\rangle+\cdots+\langle(\beta_n\alpha(q_n)-\delta)_+\rangle \nonumber\\
&&\leq\langle q_0\rangle+\langle q_1\rangle+\cdots +\langle q_n\rangle.\nonumber
\end{eqnarray}

\end{proof}

\begin{corollary}\label{thm:3.5}
 Let $\mathcal{P}$ be a class of stably finite  unital
${\rm C^*}$-algebras such that for  any $B\in \mathcal{P},$  ${\rm W}(B)$  is  almost unperforated.   Then ${\rm W}(A)$ is  almost unperforated  for  any  simple  unital ${\rm C^*}$-algebra $A\in {\rm TA}\mathcal{P}$.
\end{corollary}
\begin{proof} This is a special case of  Theorem \ref{thm:3.1} and also of  Theorem \ref{thm:3.2}.
\end{proof}

 \end{document}